\documentclass[10pt]{amsart}

\usepackage{amssymb, amsfonts, amsmath, amsthm, amsbsy}
  \ifx\luatexversion\undefined
  \ifx\XeTeXversion\undefined
\usepackage[english]{babel}
\else
\usepackage{fontspec}
\usepackage{polyglossia}
\setmainlanguage{english}
\usepackage{microtype}
\usepackage[math-style=ISO]{unicode-math}
\fi
\else
\usepackage{fontspec}
\usepackage[english]{babel}
\usepackage{lualatex-math}
\usepackage[math-style=ISO]{unicode-math}
\fi

\usepackage{hyperref}
\usepackage{todonotes, graphicx, verbatim, enumerate}
\usepackage{pagecolor}

\newcommand{\eps}{\varepsilon}

\newcommand{\eg}{\text{e.g.}\ }
\newcommand{\ie}{\text{i.e.}\ }

\newcommand{\resp}{\text{resp.}\ }

\numberwithin{equation}{section}

\usepackage{lunarized-ng}
\usepackage{todonotes}
\usepackage{xcolor}
\usepackage[normalem]{ulem}
\usepackage[T1]{fontenc}

\usepackage{amsthm}
\usepackage{amsmath}
\usepackage{hyperref}
\usepackage{dsfont} 
\usepackage{tikz}

\newcommand{\ve}{\varepsilon}
\renewcommand{\resp}[1]{(resp.~#1)}

\newcommand{\bT}{\mathbb{T}}

\newcommand{\bR}{\mathbb{R}}
\newcommand{\bZ}{\mathbb{Z}}
\newcommand{\bA}{\mathbb{A}}
\newcommand{\bC}{\mathbb{C}}
\newcommand{\hbA}{{\bA}}
\newcommand{\rrr}{h}
\newcommand{\cbv}{C_{\textrm{BV}}}
\newcommand{\ccbv}{C_{\star}}

\newcommand{\tg}{\textrm{tg}}
\newcommand{\arctg}{\textrm{arctg}}
\newcommand{\im}{\textrm{Im}\,}
\newcommand{\re}{\textrm{Re}\,}

\newcommand{\rhoal}{\rho_{\al}}
\newcommand{\rhoL}{\rho_{\El}}

\newcommand{\Cl}{C_{\El}}

\newcommand{\Ang}{Z}
\newcommand{\Aang}{A}
\newcommand{\al}{\textrm{al}}

\DeclareFontFamily{U}{wncy}{}
\DeclareFontShape{U}{wncy}{m}{n}{<->wncyr10}{}
\DeclareSymbolFont{mcy}{U}{wncy}{m}{n}
\DeclareMathSymbol{\El}{\mathord}{mcy}{"4C}
\DeclareMathSymbol{\el}{\mathord}{mcy}{"6C}

\numberwithin{equation}{section}

\title[Birkhoff Billiards inside a disc]{Yet another characterization
  of Birkhoff Billiards inside discs}
\author{Klaudiusz Czudek}
\address{Klaudiusz Czudek, Institute of Applied Mathematics, Faculty of Physics and Applied Mathematics, Gda{\'n}sk University of Technology, ul. Gabriela Narutowicza 11/12, 80-223 Gda{\'n}sk, Poland}
\email{klaudiusz.czudek@gmail.com}

\author{Jacopo De Simoi}
\address{Jacopo De Simoi\\
  Department of Mathematics\\
  University of Toronto\\
  40 St George St. Toronto, ON, Canada M5S 2E4}
\email{{\tt jacopods@math.utoronto.ca}}
\urladdr{\href{http://www.math.utoronto.ca/jacopods}{http://www.math.utoronto.ca/jacopods}}

\author{Andrew Gad}
\address{Andrew Gad,   Department of Mathematics\\
  University of Toronto\\
  40 St George St. Toronto, ON, Canada M5S 2E4}
\email{andrew.gad@utoronto.ca}
\author{Marco Poon}
\address{Marco Poon, Department of Mathematics\\
  University of Toronto\\
  40 St George St. Toronto, ON, Canada M5S 2E4}
\email{marco.poon@utoronto.ca}
\subjclass[2020]{ }


\begin{document}
\maketitle
\begin{abstract}
  In this short paper, we show a characterization of Birkhoff
  Billiards inside discs which is related to the expansion of the
  formal Lazutkin conjugacy at the boundary.
\end{abstract}
\section{Introduction and statement of our main results}
Birkhoff billiards have been introduced almost a century ago by
Birkhoff~\cite{Bf27} as an example of a dynamical system that is
simple to describe, but yet extremely rich in features.  The system
describes a particle moving inside an open domain
$\Omega\subset\bR^{2}$ along geodesics (straight lines) with the rule
that, upon collisions with the boundary $\partial\Omega$, the particle
reflects according to the law of optical reflection (\ie the angle of
incidence equals the angle of reflection).  Despite the apparent
simplicity of the problem, understanding billiard dynamics in an
arbitrary domain is extremely challenging.  Although much progress has
recently taken place (see \eg the survey~\cite{MR4374965}), many
natural questions have remained open since their introduction and are
very active topics of research.

In order to fix ideas, we assume $\partial\Omega$ to be smooth (\ie
parametrized by a $C^{\infty}$ curve) and strongly convex (\ie such
that the curvature at any point is strictly positive); we let $\ell$
denote the length of the boundary $\partial\Omega$ and
$\bT_{\ell} = \bR/\ell\bZ$.  We denote with
$(s,\varphi)\in \mathcal M = \bT_{\ell}\times [0,\pi]$ the
\emph{Birkhoff coordinates} on the phase space, where $s$ is the
arc-length parameter corresponding to a collision point and $\varphi$
is the angle that the outgoing trajectory forms with the positively
oriented tangent to the boundary at that point.  A particle following
the trajectory identified by $(s,\varphi)$ will hit the boundary
$\partial\Omega$ again at a unique point identified by the arc-length
parameter $s'$ and then reflects off the boundary with a certain angle
of reflection $\varphi'$.  The map $f:(s,\varphi)\mapsto(s',\varphi')$
is called the \emph{Billiard Map}; since the domain has a smooth
boundary, then the billiard map is also smooth.

It is elementary to show that $f$ is diffeomorphism of
$\bT_{\ell}\times[0,\pi]$ onto itself: let $\mathcal J$ denote the
diffeomorphism $(s,\varphi) \mapsto (s,\pi-\varphi)$.  Observe that
$\mathcal J$ is an involution (\ie
$\mathcal J\circ \mathcal J = \textrm{Id}$) and, moreover, that
$f\circ\mathcal J\circ f\circ \mathcal J = \textrm{Id}$, which implies
that $f^{-1} = \mathcal J\circ f\circ \mathcal J$ (the map
$\mathcal J$ is called a \emph{reversor for $f$} as it reverses the
direction of time).  This observation shows that the dynamics of $f$
in a neighbourhood of the boundary curve $\{\varphi = 0\}$ also
determines the dynamics of $f$ in a neighbourhood of the other
boundary curve $\{\varphi = \pi\}$; for this reason, when studying the
behaviour of $f$ near the boundary, one can restrict to one of the two
boundary curves.  It is customary to choose the curve
$\{\varphi = 0\}$.

If $\Omega$ is a disc (of radius $r = \ell/2\pi$), then elementary
geometrical considerations yield an explicit formula for $f$:
\begin{align}\label{eq:billiard-map-disc}
  f:(s,\varphi) = (s + 2r\varphi,\varphi).
\end{align}

Lazutkin (see \eg \cite{L}) was the first to observe that if the
domain is smooth, the billiard map can be conjugated in a
neighbourhood of the boundary $\{\varphi = 0\}$ to an arbitrarily good
approximation of~\eqref{eq:billiard-map-disc}.  There are many
slightly different versions of this fundamental result; the version we
propose below can be obtained by combining the arguments
in~\cite[Appendix B]{MR3867233} and~\cite[Appendix A]{MR3665005};
see also~\cite[Lemma 14.6]{MR1239173}.
\begin{lem}\label{lem:high-order-lazutkin}
  Assume $\partial\Omega$ is smooth, for any $k > 0$ there exists a
  neighbourhood $U$ of $\{\varphi = 0\}$ and a diffeomorphism
  $\Psi:U\to V$, where $V\subset\bT\times[0,\varepsilon)$ is a
  neighbourhood of $\bT\times\{0\}$ which conjugates $f$ with the
  normal form
  \begin{align}\label{eq:lazutkin-normal-form}
    (x,y)\mapsto(x+y,y)+O(y^{2k+2}).
  \end{align}
  We call the above normal form \emph{Lazutkin normal form of order
    $k$}.  Moreover, we can write $\Psi = (X,Y)$, where
  \begin{align}\label{def_alpha}
    X(s,\varphi) &=  \sum_{j = 0}^{k-1}X_{j}(s)\varphi^{2j}& Y(s,\varphi) &=  \sum_{j = 0}^{k}Y_{j}(s)\varphi^{2j+1}.
\end{align}
Finally, if we normalize $X$ so that $X(0,\varphi) = 0$ for every $\varphi$,
then the functions $X_{j}$ and $Y_{j}$ are uniquely determined for
$j = 0,\cdots,k-1$.
\end{lem}
\begin{rmk}
  Since a Lazutkin normal form of order $k$ is also a Lazutkin normal
  form of order $k'$ for $k' < k$, we conclude that the sequence
  $X_{j}$ and $Y_{j}$ can be uniquely determined from $\Omega$.
  Notice that $X_{0}$ is a toral diffeomorphism, whereas $X_{j}$ for
  $j > 0$ and $Y_{j}$ for $j\ge0$ are real-valued functions.
\end{rmk}

Observe that if $k = 1$, the function $X(s,\varphi) = X_{0}(s)$ simply
amounts to a reparametrization of the boundary, which is classically
known as the \emph{Lazutkin parametrization} and that we are now about
to describe.  Let $\rhoal(s)$ denote the radius of curvature of the
domain at the point parametrized by arc-length $s$ and
$\theta_{\al}(s)$ denote the angle that the positively oriented
tangent to $\partial\Omega$ forms with the horizontal positive
semi-axis.  Recall that, by definition, the radius of curvature at $s$
is the derivative of the tangent angle with respect to arc-length, that is:
\begin{align}\label{eq:theta-rho-relation}
  \dot\theta_{al}(s) &= \rhoal^{-1}(s).
\end{align}
Integrating the above along the domain yields the following elementary
version of Gauss–Bonnet formula, which will be useful in the sequel:
\begin{align}\label{sec:gauss-bonnet}
  \int_{0}^{\ell}\rhoal^{-1}(s)ds = 2\pi.
\end{align}
Lazutkin explicitly computed (in~\cite{L}) the
reparametrization formula from arc-length as follows:
\begin{align}\label{eq:lazutkin}
  X_{0}(s) = C_{\El} \int_{0}^{s}\rhoal^{-2/3}(s')ds',
\end{align}
where $C_{\El}$ is the normalizing
factor:
\begin{align}\label{eq:lazutkin-constant-arc-length}
  C_{\El} = \left[\int_{0}^{\ell}\rhoal^{-2/3}(s')ds'\right]^{-1}.
\end{align}
From the Lazutkin formula~\eqref{eq:lazutkin} it is clear that $X_{0}$
is a linear rescaling if and only if $\Omega$ is a disc (since
$X_{0}'$ is a constant if and only if $\rhoal$ is a constant).

In this paper we perform a study of $X_{1}(s)$ that  allows to prove
the following characterization of Birkhoff billiards in a disc.
\begin{mthm}
  A smooth domain $\Omega$ is a disc if and only if $X_{1} \equiv 0$.
\end{mthm}
\begin{rmk}\label{rmk:smoothness}
  In this paper (and hence the statement of our Main Theorem above ) we
  assume $C^{\infty}$ smoothness of the domain out of convenience, but
  the result indeed holds also in finite smoothness (\eg for $C^{8}$
  domains).
\end{rmk}
The formula~\eqref{eq:billiard-map-disc} together with uniqueness of
the conjugacy $\Psi$ immediately implies that if $\Omega$ is a disc,
then $X_{1} \equiv 0$; this paper concerns the proof of the reverse
implication.
\subsection{Comments and ideas for future developments}
\begin{rmk}
  Our Main Theorem is one of the many dynamical characterizations of
  billiards in disc.  For instance Bialy in~\cite{Bi} showed that the
  only domain whose phase space is completely foliated by
  non-contractible invariant curves (\ie the billiard is
  \emph{globally integrable}) is a disc.  Discs are also characterized
  by the length of periodic orbits; for instance it is possible to
  show (see~\cite{MR216370,MR293501,MR866129,MR4628019}) that if the
  maximal length of a $n$-gon inscribed in $\Omega$ equals the length
  of the regular $n$-gon inscribed in a disc of circumference $\ell$,
  then $\Omega$ is necessarily a disc.  In the recent
  preprint~\cite{arXiv:2509.06915} it is shown that knowledge of the
  Mather's $\beta$ function at only one rotation number may suffice to
  determine that the domain is a disc.
\end{rmk}
\begin{rmk}
  Lemma~\ref{lem:high-order-lazutkin} shows that it is always possible
  to formally (\ie up to $O(\varphi^{\infty})$) conjugate the billiard
  map “at the boundary” of the domain to some \emph{universal} normal
  form, which is given by
  \begin{align*}
    (\xi,\eta)\mapsto(\xi+\eta,\eta).
  \end{align*}
  The term \emph{universal} above refers to the fact that the normal
  form itself does not contain any information about the domain.  It
  is rather the conjugacy map $\Psi$ that, through this procedure,
  absorbs all information about the domain.  It makes therefore sense
  to ask how much information about the domain can be retrieved by a
  partial knowledge of the formal conjugacy, and this is the direction
  in which this paper is progressing.
\end{rmk}
\begin{rmk}[On Lazutkin conjugacy
  equivalence]\label{rmk:lazutkin-conjugacy-equivalence}.  On the one
  hand, in the remark above it is observed that, in the limit
  $k\to\infty$, the normal form is stripped of all information on
  $\Omega$.  On the other hand, for any $k$, some information is still
  present in the structure of the remainder terms
  in~\eqref{eq:lazutkin-normal-form}.  Lazutkin (see
  \eg~\cite[(14.3)]{MR1239173}) considers the following change of
  variables
\begin{align}\label{eq:classical-Lazutkin}
  X(s,\varphi) &= \Cl \int_0^s \rhoal^{-2/3}(s')ds',&
  Y(s,\varphi) &= 4\Cl \rhoal^{1/3}(s) \sin({\varphi/2}).
\end{align}
This variation on~\eqref{def_alpha} has the desirable property of
preserving a symplectic form.  The billiard map in the above
coordinates~\eqref{eq:classical-Lazutkin} admits the expansion (see
\eg~\cite[(14.4)]{MR1239173})
\begin{equation}\label{billiard_map_lazutkin}
\left\{\begin{array}{rcl}
x^+ &=& x + y + \tilde{\alpha}_3(x) y^3 + \tilde{\alpha}_4(x) y^4 + O(y^5)\\
y^+ &=& y + \tilde{\beta}_4(x)y^4 + O(y^5).
\end{array} \right.
\end{equation}
Fierobe recently announced in~\cite{Fierobe_2025} the proof that two
billiard tables with the same expansion \eqref{billiard_map_lazutkin}
are necessarily homothetic (\ie a rescaling of one another).  It can be
calculated that
\begin{align*}%
 \tilde{\alpha}_3&=\frac1{27}\left(\frac{\rhoL'}{\rhoL}\right)^{2} - \frac{1}{36}\frac{\rhoL''}{\rhoL} + \frac{1}{96 \Cl^2 \rhoL^{2/3}},\\
  \tilde{\alpha}_4&= -\frac{4}{135}\left(\frac{\rhoL'}{\rhoL}\right)^{3} + \frac{11}{270}\frac{\rhoL'\rhoL''}{\rhoL^2} - \frac1{90}\frac{\rhoL'''}{\rhoL} - \frac{\rhoL'}{360 \Cl^2 \rhoL^{5/3}},\\
\tilde{\beta}_4 &= \frac{2}{135}\left(\frac{\rhoL'}{\rhoL}\right)^{3}
                  - \frac{11}{540}\frac{\rhoL'\rhoL''}{\rhoL^2} + \frac1{180}\frac{\rhoL'''}{\rhoL} + \frac{\rhoL'}{720 \Cl^2 \rhoL^{5/3}}.
\end{align*}
Notice that $\tilde\beta_{4} = -\tilde\alpha_{4}/2$.  From our
discussion above, it is clear that there should be a relation between
$\tilde\alpha_{3}, \tilde\alpha_{4}$ and $\tilde\beta_{4}$ and the
functions appearing in the expansion~\eqref{def_alpha}.  For instance,
modifying $Y_{3}(s)$ will arbitrarily affect $\tilde\alpha_{3}$ (for
instance, Lemma~\ref{lem:high-order-lazutkin} ensures that we could
find $Y_{1}(s)$ which would set $\tilde\alpha_{3} = 0$; see
\eg~Remark~\ref{rmk:relation-with-Lazutkin} for further relations).

Fierobe subsequently realized that the proof in~\cite{Fierobe_2025} is
incorrect.  The paper uses the formulae for the billiard map which can
be found in~\cite[Page 145]{MR1239173}; unfortunately such formulae
are wrong.

The result presented in this paper would allow (see again
Remark~\ref{rmk:relation-with-Lazutkin}) to prove that the domain is a
disc if and only if $\tilde\alpha_{4} = \tilde\beta_{4} = 0$.

The result stated in~\cite{Fierobe_2025} is still without a proof; we
translate it below here in a language that is more in tune with the
setting of the present paper.\\
\textbf{Question:} Does the function $X_{1}(s)$ determine the domain
modulo homotheties?\footnote{ It is worth recalling the following much
  stronger conjecture, attributed to V. Guillemin: two billiard
  domains are homothetic if and only if their billiard maps are
  \emph{topologically conjugate}.}
\end{rmk}
\begin{rmk}[Distribution of collision points of periodic
  orbits]\label{R:alpha_DKW}
  In~\cite{MR3665005}, the authors consider axially symmetric strictly
  convex domains with boundary of class $C^r$, $r\ge 8$.  Without loss
  of generality, let us assume that the $0$ of the parametrization is
  given by one of the intersections of $\partial\Omega$ with the
  symmetry axis.  Rather elementary symmetry considerations allow to
  prove that there always exists a periodic orbit of any rotation
  number that has $0$ as a collision point; such orbit is unique for
  sufficiently small rotation number. It is proven in~\cite{MR3665005}
  that if $(x^k_q)$, $k=0,1\cdots, q-1$, are the coordinates of the
  periodic orbit of rotation number $1/q$ expressed in the Lazutkin
  parametrization~\eqref{eq:lazutkin}, then
  \begin{equation}\label{alpha_periodic}
    x^k_q = \frac k q + \frac{\alpha(k/q)}{q^2}+O(q^{-4})
  \end{equation}
  for some function $\alpha\in C^{r-4}(\mathbb{T})$.  Inspecting the
  proof (presented in~\cite[Appendix A]{MR3665005}) reveals that
  $X_1(s)=-\alpha(X_0(s))Y_0(s)^2$. Since
  $Y_0(s)=2\Cl\rho^{1/3}(s) \neq 0$, the assumption in our Main
  Theorem is equivalent to say that $\alpha$ in~\eqref{alpha_periodic}
  equals zero.  In other terms: we know that for an arbitrary
  symmetric domain, collision points of periodic orbits are
  Lazutkin-equidistributed modulo a correction
  $\alpha/q^{2}+O(1/q^{4})$.  Our Main Theorem asserts that the
  leading order ($1/q^{2}$) of this correction is $0$ if and only if
  the domain is a disc.  A version of this remark can be also
  formulated for domains that are not necessarily symmetric.
\end{rmk}
\begin{rmk}
  Of course it is natural to ask similar questions for $X_{j}$ with
  $j > 1$.  We expect that it would still be possible to find
  differential relations involving $X_{1}, X_{2}, \cdots, X_{j}$ that
  would imply that $\Omega$ is a disc.  It is unclear that such
  relations would be as simple as the one presented in this paper.
\end{rmk}
\subsection*{Acknowledgements}
This paper originated as an undergraduate summer research project for
AG and MP funded by a Mathematical and Computational Sciences Research
Award at the University of Toronto Mississauga under the supervision
of JDS.  JDS acknowledges support of the NSERC Discovery grant,
reference number RGPIN-2022-04188.  JDS also wishes to thank Fabio
Pusateri, for our always insightful exchanges of ideas, and Amir Vig,
who suggested the reference~\cite{Fierobe_2025}.  JDS and KC are
thankful to Corentin Fierobe for useful discussions about his
paper.  The algebraic computations presented in this paper have been
carried out using the free open-source mathematics software system
\href{https://www.sagemath.org/}{SageMath}, which is released under
the GNU General Public Licence GPLv2+.
\section{The main calculation}
The purpose of this section is to obtain a relation involving
$X_{1}$ and $\rhoal$; this relation will be used in the next
section to prove our Main Theorem.  One possible way to approach this
problem is by computing explicit formulas for the change of
coordinates performed in~\cite[Lemma 5.1]{MR3665005} or~\cite[Appendix
B]{MR3867233}. Here we propose an equivalent method, that is slightly
more geometric in spirit and has the additional advantage that all the
computations are performed in the same coordinates.
\subsection{Angular functions}
We will from now on assume that the arc-length parametrization of
$\partial\Omega$ is smooth (see Remark~\ref{rmk:smoothness}).  Also,
we find convenient to identify $\bR^{2}$ with $\bC$: correspondingly,
we will denote points on $\partial\Omega\subset\bC$ with the symbol
$z$.  This being established, it is useful to observe that we can
express the arc-length parametrization as:
\begin{align}\label{eq:complex-parametrization}
  z(s) = \int_{0}^{s}e^{i\theta_{\al}(s')}ds',
\end{align}
where we assume that the point identified by $s = 0$ is at the origin
and tangent to the positive horizontal direction (in particular
$\Omega$ is contained in the upper half plane).  We begin by providing
some definitions:
\begin{mydef}\label{def:pseudocollision}
  Let $(z_{-}, z_{0}, z_{+})$ identify a triple of points on
  $\partial\Omega$. Let $\varphi^{+}$ \resp{$\varphi^{-}$} denote the
  angle between $z_{0}z_{+}$ \resp{$z_{0}z_{-}$} with the positively
  \resp{negatively} oriented tangent direction at $z_{0}$ (see
  Figure~\ref{F1}).  Let $\ve > 0$: we say that $(z_{-},z_{0},z_{+})$
  identifies a \emph{billiard $\ve$-pseudocollision} if
  $|\varphi^{+}-\varphi^{-}| < \ve$.
\end{mydef}
Notice that if $(z_{-},z_{0}, z_{+})$ are three subsequent collision
points of an actual billiard trajectory, then such triple is an
$\eps$-pseudocollision for every $\eps > 0$.  Notice moreover that
$(z_{-},z_{0}, z_{+})$ is an $\eps$-pseudocollision if and only if so
is $(z_{+},z_{0}, z_{-})$.
%
%
Let us introduce the notation $\hbA_{\ve} = \bT\times(-\ve,\ve)$ and
let $\hbA_{\ve}^{*} = \hbA_{\ve}\setminus(\bT\times\{0\})$.
\begin{mydef}[Angular function]\label{def:angular-function}
  Let $k\ge0$; a smooth function
  $\Ang\in C^{\infty}(\hbA_{\ve},\partial\Omega)$ is called an
  \emph{angular function} of order $k$ if:
  \begin{enumerate}
  \item $\Ang(\cdot,\eta) = \Ang(\cdot, -\eta)$;
  \item for any $\eta\in(-\ve,\ve)$, the map $\Ang(\cdot,\eta)$ is an
    orientation-preserving diffeomorphism and $\Ang(0,\eta) = \Ang(0,0)$;
  \item for any $(\xi,\eta)\in\hbA_{\ve}^{*}$, the triple
    \begin{align*}
      (\Ang(\xi-\eta,\eta),\Ang(\xi,\eta),\Ang(\xi+\eta,\eta))
    \end{align*}
    identifies a billiard $O(\eta^{2k+2})$-pseudocollision (here the
    notation $O(\cdot)$ is intended as $\eta\to0$).
  \end{enumerate}
\end{mydef}
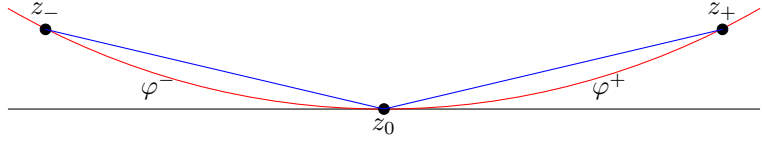
\begin{figure}[!ht]
\begin{tikzpicture}
\draw (-5,0) -- (5,0);
\draw[red, rotate=0] (0:0) arc (270:300:10);
\draw[red, rotate=0] (0:0) arc (270:240:10);
\filldraw (4.5,1.06) circle(2pt) node[anchor=south]{$z_+$};
\filldraw (-4.5,1.06) circle(2pt) node[anchor=south]{$z_-$};
\filldraw (0,0) circle(2pt) node[anchor=north]{$z_0$};
\draw[blue] (4.5, 1.06) -- (0,0);
\draw[blue] (-4.5, 1.06) -- (0,0);
\node at (3,0.3) {$\varphi^+$};
\node at (-3,0.3) {$\varphi^-$};
\end{tikzpicture}
\caption{The triple $(z_{-}, z_{0}, z_{+})$ identifies a billiard
  $\ve$-pseudocollision if $|\varphi^{+}-\varphi^{-}| < \ve$}
\label{F1}
\end{figure}

Let $\Ang$ be an angular function and define
$\Ang^{\pm}(\xi,\eta) = \Ang(\xi\pm \eta,\eta)$.  For
$(\xi,\eta)\in\hbA_{\ve}^{*}$ we define the functions
$\Phi^{\pm}(\xi,\eta)$ as follows:\footnote{ The choice to use
  \emph{arctangent} rather than \emph{argument} in the definition
  is deliberate, in order to obtain convenient symmetry properties.
  An alternative approach would be to use a projective version of the
  argument function (\ie the argument modulo $\pi$ so that $-\varphi$
  and $\pi-\varphi$ would then be equivalent). }
\begin{align}\label{eq:Phi-pm}
  \Phi^{\pm}(\xi,\eta) &= \arctg\frac{\im \delta z^{\pm}(\xi,\eta)}{\re
                         \delta z^{\pm}(\xi,\eta)}
,\textrm{where }\delta z^{\pm}(\xi,\eta) = e^{-i\theta(\Ang(\xi,\eta))}
                         \left[{\Ang^{\pm}(\xi,\eta)}-\Ang(\xi,\eta)\right],
\end{align}
and $\theta(z)$ denotes the angle between the positively oriented
tangent vector to $\partial\Omega$ at $z$ and the positive horizontal
direction.  Notice that, since we chose to preserve orientation in
Definition~\ref{def:angular-function}, for positive $\eta$, the
complex number $\delta z^{+}(\xi,\eta)$
\resp{$\delta z^{-}(\xi,\eta)$} belongs to the top-right
{\resp{top-left}} quadrant.  It is hence clear that, for positive
$\eta$, the functions $\Phi^{\pm}({\xi,\eta})$ give the angles
$\varphi^{\pm}$ corresponding to the triple
$(\Ang(\xi-\eta,\eta),\Ang(\xi,\eta),\Ang(\xi+\eta,\eta))$.  Observe
that since $\partial\Omega$ is smooth, we can guarantee that
$\theta(s)$ is also smooth.  Notice that our definition implies
$\Ang^{\pm}(\xi,-\eta) = \Ang^{\mp}({\xi,\eta})$, hence
$\Phi^{\pm}(\xi,\eta) = -\Phi^{\mp}(\xi,-\eta)$.  Moreover,
$\lim_{\eta\to0}\Phi^{\pm}(\xi,\eta) = 0$; we thus extend $\Phi^{\pm}$
to the whole annulus $\hbA_{\ve}$ by continuity, by setting
$\Phi^{\pm}(\cdot,0) = 0$.

We claim that, with this definition, $\Phi^{\pm}$ is a smooth
function.  In fact it is clear that $\Phi^{\pm}$ is smooth on
$\hbA_{\ve}^{*}$; it is then a simple exercise\footnote{ The exercise
  is similar to checking that if a smooth function $f(x)$ goes to $0$
  at $x = 0$, then $f(x)/x$ (extended by continuity to be equal to
  $f'(0)$ at 0) is also smooth.}  to show that $\Phi^{\pm}$ is also
smooth at $\eta = 0$.

The above discussion implies that $\Phi^{+}+\Phi^{-}$ is an odd
function of $\eta$, whereas $\Phi^{+}-\Phi^{-}$ is an even function of
$\eta$.  This observation justifies the fact that we only consider
even powers of $\eta$ in item (3) in
Definition~\ref{def:angular-function}.
\begin{rmk}\label{rmk:even-terms-angular}
  The above discussion suggests a condition equivalent to item (3) in
  Definition~\ref{def:angular-function}.  Let us consider the
  (partial) Taylor expansion at $\eta = 0$:
   \begin{align}\label{eq:Phi-pm-expansion}
    \Phi^{\pm}(\xi,\eta) = \sum_{j = 1}^{2k+1}\Phi^{\pm}_{j}(\xi){\eta^{j}}+O(\eta^{2k+2}).
  \end{align}
  The symmetry relations discussed above imply that the following
  facts are equivalent
  \begin{itemize}
  \item $\Phi^{+}_{2l} \equiv 0$ for $l\in\{1,\cdots,k\}$;
  \item $\Phi^{-}_{2l} \equiv 0$ for $l\in\{1,\cdots,k\}$;
  \item $\Ang$ is an angular function of order $k$.
  \end{itemize}
  The relations appearing in any of the first two items will provide
  (differential) equations that can be solved to obtain explicit
  expressions for the function $\Ang$.
\end{rmk}
\subsection{Relation between angular functions and high-order Lazutkin
  coordinates}
We now exhibit a relation between the function $\Ang$ and the Lazutkin
change of variables $\Psi$: more in detail, we show that the existence
of an angular function $\Ang$ of order $k$ is equivalent to existence
of a change of coordinates $\Psi_{k}$ given in
Lemma~\ref{lem:high-order-lazutkin}.  We call such a change of
coordinates \emph{Lazutkin change of coordinates of order $k$}.
\begin{lem}\label{lem:angular-lazutkin} The following are equivalent:
  \begin{enumerate}
  \item there exists a Lazutkin change of coordinates of
    order $k$;
  \item there exists a angular function $\Ang$ of order $k$
    on some $\hbA_{\ve}$.
  \end{enumerate}
\end{lem}
\begin{proof}
  Let us begin by proving the implication (1)$\Rightarrow$(2); let $U$
  denote the neighbourhood of $\{\varphi = 0\}$ where $\Psi$ is
  defined and invertible.  The formulas~\eqref{def_alpha} imply that
  $\Psi(\{\varphi = 0\}) = \{y = 0\}$ and, moreover, allow $\Psi$ to
  be extended to negative values of $\varphi$ in such a way that
  $\Psi(s,\varphi) = J\circ\Psi(s,-\varphi)$, where
  $J:(x,y)\mapsto(x,-y)$.  Such extension (defined on the symmetrized
  neighbourhood $\hat U$) is also smooth.  By continuity there exists
  $\ve$ so that $\Psi(\hat U)\supset\hbA_{\ve}$.  We can thus define
  $\Ang$ on $\hbA_{\ve}$ to be
  $\gamma\circ\pi_{s}\circ\Psi^{-1}(\xi,\eta)$, where $\pi_{s}$ is the
  projection on the first (arc-length) coordinate and $\gamma$ is the
  arc-length parametrization of $\partial\Omega$.  We claim that
  $\Ang$ is an angular function of order $k$.  Items (1) and (2) are
  immediate; by symmetry it suffices to show item (3) only for
  $\eta > 0$: by design, for any $(\xi,\eta)\in\hbA_{\ve}$, the point
  $\Psi^{-1}(\xi+\eta,\eta)$ and $f\circ\Psi^{-1}(\xi,\eta)$ are
  $O(\eta^{2k+2})$-close, and the same holds for
  $\Psi^{-1}(\xi-\eta,\eta)$ and $f^{-1}\circ\Psi^{-1}(\xi,\eta)$.  By
  smoothness of the billiard map, we conclude that the triple
  $(\Ang(\xi-\eta,\eta),\Ang(\xi,\eta),\Ang(\xi+\eta,\eta))$
  identifies a billiard $O(\eta^{2k+2})$-pseudocollision.

  Let us now prove that (2) $\Rightarrow$ (1).  It is convenient to
  write $\Ang$ in coordinates: let
  $\Ang_{\al}:\hbA_{\ve}\to\bT_{\ell}$ be the function mapping
  $(\xi,\eta)$ to the arc-length parametrization of $\Ang(\xi,\eta)$;
  without loss of generality we can assume that $\Ang_{\al}(0,\eta) = 0$.
  The argument used above item suggests that $\Ang_{\al}$ should be
  related to the first component of the inverse of the conjugacy
  $\Psi$.  Our strategy to prove indeed to construct a function
  $\Phi(\xi,\eta)$ in such a way that
  $(\Ang_{\al}(\xi,\eta),\Phi(\xi,\eta))$ is a smooth diffeomorphism on
  $\hbA_{\ve}$ and work from there.  The definition of the function
  $\Phi$ it not unique; we propose a definition that manifestly
  preserves our symmetries.  We take
  \begin{align*}
    \Phi(\xi,\eta) = \frac{\Phi^{+}(\xi,\eta)+\Phi^{-}(\xi,\eta)}{2}.
  \end{align*}
  It was already observed that $\Phi$ is an odd function of $\eta$,
  \ie $\Phi\circ J = -\Phi$; we now need to prove that
  $(\xi,\eta)\mapsto(\Ang_{\al}(\xi,\eta),\Phi(\xi,\eta))$ is
  invertible in a suitable neighbourhood of $\{\eta = 0\}$.  Since
  $\Ang_{\al}(\cdot,\eta)$ is an orientation-preserving
  diffeomorphism, we conclude that
  $\frac{\partial\Ang_{\al}}{\partial\xi} > 0$ everywhere.  Moreover,
  since $\Ang_{\al}(\cdot,\eta) = \Ang_{\al}(\cdot,-\eta)$, we
  conclude that $\frac{\partial\Ang_{\al}}{\partial\eta}|_{\eta = 0} = 0$.
  In order to conclude invertibility in some neighbourhood of
  $\{\eta = 0 \}$, it is then enough to show that
  $\frac{\partial\Phi}{\partial\eta}|_{\eta = 0}\neq 0$; it is not
  difficult to check (see e.g.~\cite[(14.1b)]{MR1239173}) that the
  following estimate holds everywhere (here $\rho_{\min}$
  {\resp{${\rho_{\max}}$}} denotes the minimum {\resp{maximum}} of
  $\rhoal$):
  \begin{align*}
    2\rho_{\min}|\Phi^{\pm}|\le |\Ang_{\al}^{\pm}-\Ang_{\al}|\le 2\rho_{\max}|\Phi^{\pm}|.
  \end{align*}
  Hence, if we can show that
  $\pm\frac{\partial\Ang_{\al}^{\pm}}{\partial\eta} > 0$, we can conclude
  that $\frac{\partial\Phi^{\pm}}{\partial\eta} > 0$.  By definition,
  \begin{align*}
    \frac{\partial\Ang_{\al}^{\pm}}{\partial\eta} =
    \pm\frac{\partial\Ang_{\al}}{\partial\xi} + \frac{\partial\Ang_{\al}}{\partial\eta}.
  \end{align*}
  If $\eta = 0$, once again we have
  $\frac{\partial\Ang_{\al}}{\partial\eta} = 0$, which then proves our
  claim since $\frac{\partial\Ang_{\al}}{\partial\xi} > 0$ everywhere.
  The map $(\Ang_{\al}(\xi,\eta),\Phi(\xi,\eta))$ is therefore a smooth
  orientation-preserving diffeomorphism in a neighbourhood of
  $\{\eta = 0\}$. Its image contains a neighbourhood $\tilde U$ of
  $\{\varphi = 0\}$ in phase space; let $\tilde\Psi$ be its inverse in
  such neighbourhood.  Then $\tilde\Psi$ has the correct symmetries to
  be written in the form~\eqref{def_alpha}, but possibly it has extra
  terms in the expansion; we define $\Psi = (X,Y)$ to be the
  truncation of $\tilde\Psi$ of its Taylor expansion up to terms of
  order $\varphi^{2k}$ (for $X$) and order $\varphi^{2k+1}$ (for $Y$);
  the function $\Psi$ will be a diffeomorphism on a possibly smaller
  neighbourhood of $\{\varphi = 0\}$ that we call $U$.  Since $\Ang$
  was an angular function, $\Phi(\xi,\eta)$ and $\Phi^{\pm}(\xi,\eta)$
  differ by $O(\eta^{2k-2})$, which implies that
  $\Psi\circ f^{\pm 1}(s,\varphi) = (X(s,\varphi)\pm
  Y(s,\varphi),Y(s,\varphi))+O(\eta^{2k+2})$, proving that $\Psi$ is a
  high-order Lazutkin change of coordinates.
\end{proof}
If $\Ang$ is an angular function of order $k$, extracting the
$2k-1$-jet of $\Ang_{\al}$ at $\eta = 0$ we obtain, in particular,
the following expansion:
\begin{align*}
  \Ang_{\al}(\xi,\eta) &= \sum_{j = 0}^{k-1} \Ang_{\al,j}(\xi)\eta^{2j}+O(\eta^{2k}).
\end{align*}
It is clear by Lemma~\ref{lem:angular-lazutkin} that there should be a
relation between the functions $\Ang_{\al,j}$ and the functions
$X_{j}$ and $Y_{j}$ given in~\eqref{def_alpha}.  For
instance, we immediately obtain:
\begin{align}\label{E:X1}
  \Ang_{\al,0}\circ X_{0} &= \text{Id}&
  \Ang_{\al,1} = - \frac{X_{1}}{\dot X_{0}Y_{0}^{2}}\circ \Ang_{\al,0}.
\end{align}
In principle, one can obtain algebraic relations for arbitrary $k$
between the functions $\Ang_{\al,k}$ and $X_{j}$ and $Y_{j}$ for
$0\le j\le k$.  The immediate observation that stems from the above
relation is that $X_{1} \equiv 0$ if and only if
$\Ang_{\al,1} \equiv0$.  We can thus restate our Main Theorem as follows:
\begin{thm}
  The domain $\Omega$ is a disc if and only if $\Ang_{\al,1} \equiv 0$.
\end{thm}
In order to further simplify the computations that will follow, it
turns out to be convenient to express $\Ang$ in Lazutkin coordinates.
Let $\Ang_{\El}:\hbA_{\ve}\to\bT$ be the function that maps
$(\xi,\eta)$ to the Lazutkin coordinate corresponding to the point
$\Ang(\xi,\eta)$.  Of course we can write a new expansion:
\begin{align*}
  \Ang_{\El}(\xi,\eta) &= \sum_{j = 0}^{k-1}\Ang_{\El,j}(\xi)\eta^{2j}+O(\eta^{2k}).
\end{align*}
The advantage in the above expression is that~\eqref{eq:lazutkin}
implies that
$\Ang_{\El,0} = X_{0}\circ\Ang_{\al,0} = \text{Id}$, hence:
\begin{align*}
  \Ang_{\El}(\xi,\eta) &= \xi + \sum_{j = 1}^{k-1}\Ang_{\El,j}(\xi)\eta^{2j}+O(\eta^{2k}).
\end{align*}
Observe that $\Ang_{\al,1}\equiv 0$ if and only if
$\Ang_{\El,1}\equiv 0$.  According to
Remark~\ref{rmk:even-terms-angular}, in order to find an equation
involving $\Ang_{\El,1}$, we need to set to zero the even terms of the
expansion~\eqref{eq:Phi-pm-expansion}.  This is the content of the
following lemma.  Let $\rhoL:\bT\to\bR_{ > 0}$ denote the radius of
curvature of $\partial\Omega$ in the Lazutkin parametrization.  We
will denote with a \emph{prime} differentiation with respect to the
Lazutkin parameter (recall that a \emph{dot} denoted differentiation
with respect to arc-length).  In order to simplify the notation we
also denote by $\Aang = \Ang_{\El,1}$.  It is also useful to rewrite
some formulas in the Lazutkin parametrization.  It is a simple
exercise to check that~\eqref{eq:complex-parametrization} can be
written as (recall~\eqref{eq:lazutkin}):
\begin{align}\label{eq:complex-lazutkin-param}
  z(x) = \frac1{\Cl}\int_{0}^{x}e^{i\theta_{\El}(x')}\rhoL^{2/3}(x')dx',
\end{align}
where $\theta_{\El}$ is the angle $\theta$ expressed in Lazutkin
parametrization; notice that the formula corresponding
to~\eqref{eq:theta-rho-relation} reads:
\begin{align}
  \label{eq:theta-rho-relation-lazutkin}
  \theta_{\El}' = \frac1{\Cl}\rhoL^{-1/3};
\end{align}
in particular, the Gauss–Bonnet formula reads $\int_{\bT}\rhoL^{-1/3} = 2\pi\Cl$.
\begin{lem}\label{P:alpha}
Recall the expansion:
\begin{align*}
  \Phi^{+}(\xi,\eta) = \sum_{j = 1}^{2k+1}\Phi_{j}^{+}(\xi)\eta^{j}+o(\eta^{k}).\end{align*}
Then
\begin{align*}
  \Phi^{+}_{2} &=  0\\
  \Phi^{+}_{4} &=
  \frac{1}{4 C_{\El}\rhoL^{1/3}}\cdot\left(
  \Aang''
  +\frac{8}{135} \left(\frac{\rhoL'}{\rhoL}\right)^{3}
  -\frac{11}{135} \frac{\rhoL'\rhoL''}{\rhoL^2}
  +\frac{1}{45} \frac{\rhoL'''}{\rhoL}
  +\frac{1}{180 \Cl^2} \frac{\rhoL'}{\rhoL^{5/3}}\right),
\end{align*}
where all the functions are evaluated in (and differentiated with
respect to) the variable $\xi$.
\end{lem}
\begin{proof}
  Let us recall the definition of $\Phi^{+}$ from~\eqref{eq:Phi-pm};
  we now rewrite the complex vector $\delta z^{\pm}$
  using~\eqref{eq:complex-lazutkin-param}
  and~\eqref{eq:theta-rho-relation-lazutkin} obtaining:
\begin{align}\label{eq:Phi-pm-lazutkin}%
  \delta z^{\pm}(\xi,\eta) =
  \frac{e^{-i\theta_{\El}(\Ang_{\El}(\xi,\eta))}}{C_{\El}}\int_{\Ang_{\El}(\xi,\eta)}^{\Ang_{\El}^{\pm}(\xi,\eta)}e^{i\theta_{\El}(x)}\rhoL^{2/3}(x)dx.
\end{align}
Since $\Ang_{\El,0} = \text{Id}$ and we are in Lazutkin
parametrization, we conclude that $\Ang_{\El,0}$ is an angular
function of order $1$, hence $\Phi^{+}_{2} = 0$ (this follows also
from an explicit computation).

Since moreover $\arctg$ is an odd function of its argument, it holds
that $\Phi_{4}^{+}$ equals
the fourth Taylor coefficient of $\tg\Phi^{+}$; in other words, we can
avoid taking the arctg in order to compute the required coefficient.

Applying Lemma~\ref{L:appendix} we can thus write:
\begin{align*}
  \tg\Phi^{+}(\xi,\eta) =\frac{S_2\eta^{2}+S_3\eta^{3}+S_4\eta^{4}+S_{5}\eta^{5}}{C_1\eta+C_2\eta^{2}+C_3\eta^{3}+C_4\eta^{4}}+O(\eta^{5})
\end{align*}
The fourth-order term in this expansion is thus given by (see once
again Lemma~\ref{L:appendix} for the definitions of the functions
$\tilde S_{j}$ and $\tilde C_{j}$):
\begin{align*}
  \Phi^{+}_{4}&= C_{\El}^{-1}\rho^{-1/3}\cdot\left(\tilde S_{5}-\tilde
                C_{2}\tilde S_{4}+(\tilde C_{2}^{2}-\tilde
                C_{3})\tilde S_{3}-(\tilde C_{2}^{3}-2\tilde
                C_{2}\tilde C_{3}+\tilde C_{4})\tilde S_{1}\right).
\end{align*}
Carrying out the algebra yields the desired formula.
\end{proof}
We conclude that, in order for $\Ang_{\El}$ to be an angular function
of order $2$, the function $\Aang = \Ang_{\El,1}$ solves the following
differential equation:
  \begin{align}\label{eq:alpha-ODE}
    \Aang''
    +\frac{8}{135} \left(\frac{\rhoL'}{\rhoL^3}\right)^{3}
    -\frac{11}{135} \frac{\rhoL'\rhoL''}{\rhoL^2}
    +\frac{1}{45} \frac{\rhoL'''}{\rhoL}
    +\frac{1}{180 \Cl^2} \frac{\rhoL'}{\rhoL^{5/3}}=0.
  \end{align}
  with periodic boundary conditions and (by our choice of
  normalization) $\Aang(0) = 0$.  We will study solutions the
  differential equation in the next section.
  \begin{rmk} Among the four terms involving $\rho$, three have
    differential degree $3$ and the remaining term has differential
    degree $1$ (once renormalized with $\Cl^{2}\rho^{2/3}$).  Terms
    with different differential degrees come from different terms in
    the expansion of the complex exponential function
    in~\eqref{eq:Phi-pm-lazutkin}.
  \end{rmk}
  \begin{rmk}\label{rmk:relation-with-Lazutkin}
    Observe that~\eqref{eq:alpha-ODE} is equivalent to the equation
    \begin{align*}
      \frac14A''+ \tilde\beta_{4} = 0,
    \end{align*}
    where $\tilde\beta_{4}$ is defined in
    Remark~\ref{rmk:lazutkin-conjugacy-equivalence}.  This is of
    course to be expected, since $\tilde\beta_{4}$ is the correction
    of order $4$ to the angular component $y$ of the billiard map in
    classical Lazutkin coordinates, and $A$ is exactly designed to
    compensate such correction.
  \end{rmk}
\section{Proof of the Main Theorem}
In this section we prove the Main Theorem; we begin by some elementary
manipulations
\subsection{Solving the differential equation~\eqref{eq:alpha-ODE}}
Performing the convenient substitution $\rrr=
\rhoL^{-1/3}/2\sqrt2$
in~\eqref{eq:alpha-ODE}, we obtain:
\begin{align}\label{eq:ODE-third-degree}
  \nonumber \Aang'' &= \frac1{15}\frac{\rrr'''}{\rrr} -
  \frac1{15}\frac{\rrr'\rrr''}{\rrr^2} + \frac{2}{15\Cl^2}\rrr\rrr' = \\
&= \frac1{15} \left[\frac{h''}{h}+\frac{h^{2}}{\Cl^{2}}\right]'.
\end{align}
Observe that, with this substitution, the Gauss–Bonnet
formula~\eqref{sec:gauss-bonnet} amounts to a normalization condition
for $h$ that we write as follows:
\begin{align}\label{eq:lazutkin-constant}
  2\sqrt{2}\int_{\bT}\rrr d\xi &= 2\pi \Cl.
\end{align}
Integrating once the above differential equation, we therefore obtain
\begin{align}\label{eq:alpha-prime}
  \Aang' =  \frac1{15}
  \left[\frac{h''}{h}+\frac{\rrr^{2}}{\Cl^{2}} - \cbv\right],
\end{align}
where $\cbv$ is an integration constant that can be fixed in such a
way to ensure periodicity of $\Aang$ — that is, we set $\int\Aang' = 0$,
which yields, after integrating by parts, the manifestly positive
quantity:
\begin{align}\label{eq:cbv}
  \cbv &= \int_{\bT}\left[\left(\frac{h'}{h}\right)^{2}+\frac{h^{2}}{\Cl^{2}}\right]d\xi.
\end{align}
Thus, integrating one more time and imposing the condition
$\Aang(0) = 0$, we conclude:
 \begin{equation}
\label{E:5.1} 15\,\Aang(\xi) = \int_0^\xi\frac{\rrr''(\xi')}{\rrr(\xi')} d\xi' + \frac{1}{\Cl^2} \int_0^{\xi} \rrr(\xi')^2 d\xi' - \cbv\xi,
\end{equation}
\begin{rmk}
  It may be useful for future work to explicitly write the
  expression~\eqref{eq:alpha-prime} in terms of the curvature
  function: here derivatives are with respect to Lazutkin coordinates:
  \begin{align*}
    A' = \frac4{135}\left(\frac{\rhoL'}{\rhoL}\right)^{2}-\frac1{45}\frac{\rhoL''}{\rhoL}+\frac1{120\Cl\rhoL^{2/3}}-\frac{\cbv}{15}.
  \end{align*}
\end{rmk}
\begin{rmk}
  It is instructive to observe that the quantity $\cbv$ has some
  intrinsic meaning.  Reverting in~\eqref{eq:cbv} the substitution for
  $\rhoL$ yields
  \begin{align*}
    \cbv = \frac1{72}\int_{\bT}8\left(\frac{\rhoL'}{\rhoL}\right)^{2}+\frac9{\Cl^{2}}\rhoL^{-2/3}d\xi;
  \end{align*}
changing variables to arc-length and
recalling~\eqref{eq:lazutkin-constant-arc-length} gives:
  \begin{align*}
    \cbv = \frac1{72}\int_{\bT_{\ell}}\rhoal^{-2/3}ds\cdot\int_{\bT_{\ell}}\frac{8\dot\rhoal^{2}+9}{\rhoal^{4/3}}ds.
  \end{align*}
  The two integrals appearing in the above expression are (multiples
  of) two of the so-called Marvizi–Melrose invariants
  $\mathcal{I}_{1}$ (which we called $\Cl^{-1}$)
  and $\mathcal{I}_{2}$ (see
  \eg~\cite[(4.6)]{MMelose}, where they are defined in terms of the
  curvature $\kappa = \rhoal^{-1}$ or~\cite[Section 2.5]{MR3392661}).
  Such quantities are Laplace-spectral and Marked-Length-Spectral
  invariants.

  This perhaps surprising relation is of course not accidental.  As
  remarked in Remark~\ref{R:alpha_DKW}, the function $A$ appears (in
  some form) in the asymptotic distribution of collision points of
  periodic orbits; this in turn is tightly related with the
  asymptotics of the length of periodic orbits of small rotation
  number –or, equivalently, the coefficients of the Mather's $\beta$
  function at $0$– and it is well known (see once again~\cite[Theorem
  5.15]{MMelose} and~\cite{MR3392661}) that such coefficients are
  algebraically related with the Marvizi–Melrose invariants.
\end{rmk}

\subsection{Characterizing the disc}
We now assume that $\Aang\equiv 0$ and write an equation for $\rrr$;
substituting $\Aang''\equiv0$ in~\eqref{eq:ODE-third-degree} we obtain
the equation:
\begin{align*}
  \left[\frac{\rrr''}{\rrr}+\frac{\rrr^{2}}{\Cl^{2}}\right]' = 0.
\end{align*}
Observe that the constant $\Cl$ appearing in the above equation itself
depends on the solution $\rrr$ via the
relation~\eqref{eq:lazutkin-constant}; since $\Cl$ amounts to a
normalization, this complication can be avoided by writing the
equation in terms of the unrenormalized Lazutkin parameter.  In other
terms, performing the substitution $t = \Cl\xi$ (and denoting the
derivative with respect to $t$ with a subscript) we obtain:
\begin{align*}
  \Cl^{-3}\left[\frac{h_{tt}}{h}+h^{2}\right]_t = 0,
\end{align*}
in which the dependence on $\Cl$ can be factored out.  We now
integrate once with respect to $t$ and multiply by $\rrr$ (since
$\rrr \neq 0$) obtaining:
\begin{equation}
\label{E:6.1}
    \rrr_{tt} = -\rrr^3 + \ccbv\rrr,
\end{equation}
where $\ccbv$ is a constant of integration.  Any \emph{positive}
solution of~\eqref{E:6.1} (depending on the parameter $\ccbv$)
identifies a curvature function $\rho(t) = h(t)^{-3}/8$; in turn any
curvature function identifies a planar curve $\Gamma(t)$ by (the
unrenormalized versions of)~\eqref{eq:complex-lazutkin-param}
and~\eqref{eq:theta-rho-relation-lazutkin}:
\begin{align*}
  \Gamma(t) = \int_{0}^{t}e^{i\int_{0}^{t'}\rho^{-1/3}(t'')dt''}\rho^{2/3}(t')dt'.
\end{align*}
We will declare a solution $\rrr$ to be \emph{admissible} if $\Gamma$
is a Jordan curve (\ie the boundary of a domain).  In particular both
$\Gamma$ and $h$ need to be periodic with the same period (although
the primitive period of $\Gamma$ may be a multiple of the primitive
period of $h$).  Additionally $\rrr$ needs to satisfy the Gauss–Bonnet
formula:
\begin{align}\label{eq:last-gauss-bonnet}
  2\sqrt2\oint hdt = 2\pi,
\end{align}
where $\oint$ denotes the integral on the period of
$\Gamma$.\footnote{ Of course the mere periodicity of $h$
  and~\eqref{eq:last-gauss-bonnet} are only necessary for $\Gamma$ to
  be Jordan but not sufficient. }  We will conclude this section with
the proof of the following lemma, from which our Main Theorem ensues.
\begin{lem}\label{lem:uniqueness}
  For any $\ccbv > 0$, the differential equation~\eqref{E:6.1} has a
  unique admissible periodic solution: this solution is given by the
  constant $h \equiv \sqrt{\ccbv}$.
\end{lem}
\begin{proof}
  First of all, it is obvious to check that the constant solution
  $\rrr\equiv\sqrt{\ccbv} $ satisfies~\eqref{E:6.1}. In that case
  $\Gamma$ is a circle (of radius $\ccbv^{-3/2}/8$) and $\rrr$ is
  certainly admissible.

  Before continuing with the proof, it is first convenient to
  normalize~\eqref{E:6.1} one last time, so that $\ccbv$ does not
  appear in the equation.  Let $\rrr = \sqrt{\ccbv}\zeta$ and
  $\tau = \sqrt{\ccbv}t$ so that~\eqref{E:6.1} reads:
  \begin{align}\label{eq:normalized-ode}
    \zeta_{\tau\tau} = -\zeta^{3}+\zeta.
  \end{align}
  Observe that with this substitution, the Gauss–Bonnet
  constraint~\eqref{eq:last-gauss-bonnet} still reads
  \begin{align}\label{eq:normalized-constraints}
    2\sqrt2\oint\zeta d\tau = 2\pi.
  \end{align}
  The proof of the lemma will be complete if we can show that any
  non-constant periodic solution $\zeta$ of~\eqref{eq:normalized-ode}
  fails to satisfy~\eqref{eq:normalized-constraints}.  Let us consider the
  following function of $\zeta$ and $\zeta_{\tau}$:
  \begin{align*}
    \mathcal H(\zeta,\zeta_{\tau}) := \frac12\zeta_{\tau}^2 +
    \frac{\zeta^{4}}{4} - \frac{\zeta^{2}}{2} +\frac14.
  \end{align*}
  An elementary computation shows that $\mathcal H$ is a conserved
  quantity (in fact, $\mathcal H$ is a Hamiltonian function for the
  system given by \eqref{E:6.1}).  We thus study the level sets
  $\mathcal H = E$, which correspond to possible trajectories
  of~\eqref{E:6.1}.  It is elementary to check that such level sets are
  symmetric with respect to the transformation $\zeta\mapsto -\zeta$;
  moreover the system admits trajectories for any $E \ge 0$; such
  trajectories are always bounded.  If, moreover, $E < 1/4$, such
  trajectories do not traverse $\zeta = 0$.  Observe moreover that
  $E = 0$ corresponds to the constant solution $\zeta\equiv 1$ discussed
  earlier; since we can only admit positive solutions, we will therefore
  assume $E\in(0,1/4)$.

  We can thus solve for $\zeta_{\tau}$, obtaining:
  \begin{align*}
    \zeta_{\tau} = \pm\sqrt{2E - \frac12- \frac{\zeta^{4}}{2} +\zeta^2} =
    \pm\sqrt{2E - \frac{(\zeta^2 - 1)^{2} }{2}}.
  \end{align*}
  This shows that the integral curves of~\eqref{E:6.1} in the plane
  $(\zeta,\zeta_{\tau})$ are closed and symmetric with respect to the axis
  $\zeta_{\tau} = 0$.  The trajectory with energy $E$ oscillates between the two
  positive solutions of the bi-quadratic equation $4E = (\zeta^2 -1)^{2}$,
  which satisfy the relation
  \begin{align*}
    \zeta^{2}_{\pm} = 1 \pm 2\sqrt{E}.
  \end{align*}
  Any fixed $E\in(0,1/4)$ corresponds to a trajectory that is
  periodically oscillating back and forth between $\zeta_{-}(E)$ and
  $\zeta_{+}(E)$; let us denote its prime period by $T(E)$. 

  For given $E$, let us now define the integral
  \begin{align*}
    \Theta(E) 
    = 2\sqrt2\int_{0}^{T(E)} \zeta d\tau.
  \end{align*}
  As observed earlier, this quantity measures the angle between the
  vectors $\Gamma_{\tau}(0)$ and $\Gamma_{\tau}(T(E))$, tangent to the
  curve $\Gamma$ respectively at $\tau = 0$ and $\tau = T(E)$.
  The function $\Theta(E)$ can be computed explicitly (the crucial
  factor 2 accounts for the fact that, in one period, the function
  oscillates from $\zeta_{-}$ to $\zeta_{+}$ and back):
  \begin{align*}
    \Theta(E) = 2\cdot 2\sqrt2\int_{\zeta_{-}}^{\zeta_{+}}\frac{\zeta}{\sqrt{2E - \frac{1}{2}(\zeta^2 -1)^{2}}}d\zeta.
  \end{align*}
  Performing the  substitutions $u = \zeta^2 - 1$ and $v = u/(2\sqrt{E})$
  yields:
  \begin{align*}
    \Theta(E) = 2\sqrt2
    \int_{-2\sqrt{E}}^{2\sqrt{E}}\frac{du}{\sqrt{2E - \frac12u^2}}
    = 4\int_{-1}^{1}\frac{dv}{\sqrt{1-v^{2}}} = 4\pi.
  \end{align*}
  Remarkably, the quantity $\Theta(E)$ is independent of $E$, and
  moreover it is twice $2\pi$.  As observed earlier, the period of
  $\Gamma$ is necessarily a multiple of the period $T(E)$ of $\zeta$:
  we conclude that the left hand side
  of~\eqref{eq:normalized-constraints} equals a positive multiple of
  $4\pi$, hence it never satisfies the equation.  We conclude that no
  non-constant periodic solution is admissible.
\end{proof}
\begin{rmk}
  In the above proof we show that non-constant periodic solutions
  of~\eqref{eq:alpha-ODE} with $A'' = 0$ are curvature functions of
  curves that necessarily self-intersect.  In fact, there are
  additional considerations that would rule the solutions
  inadmissible.  For instance it would possible be to show that for
  any $E\in(0,1/4)$ the curves $\Gamma$ corresponding to the value $E$
  are not closed (see Figure~\ref{fig:data} below for some numerical
  solutions)
  \begin{figure}[!ht]
    \centering
    \includegraphics[width=0.30\textwidth, angle=0]{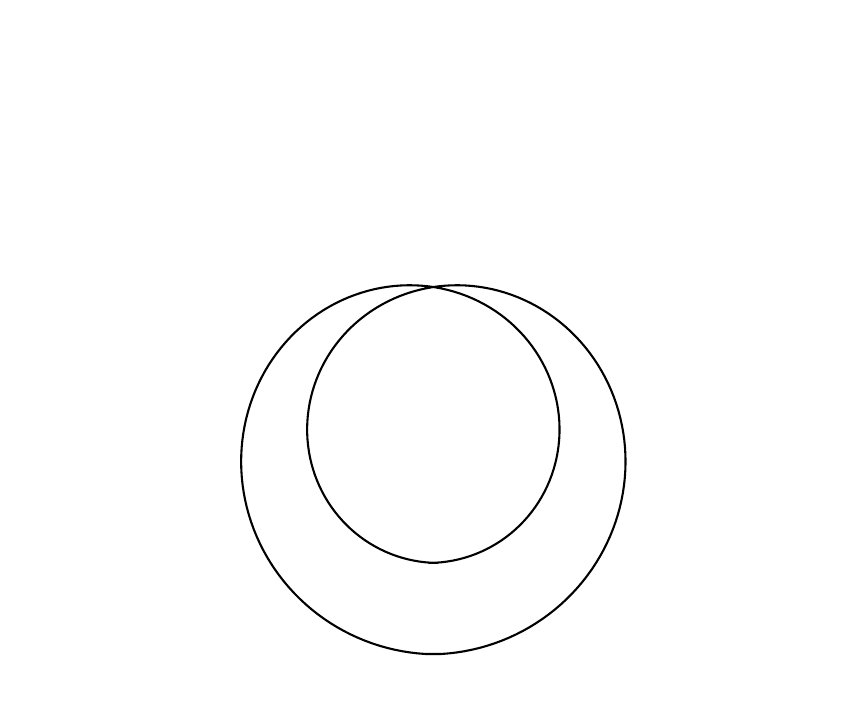}
    \includegraphics[width=0.30\textwidth, angle=0]{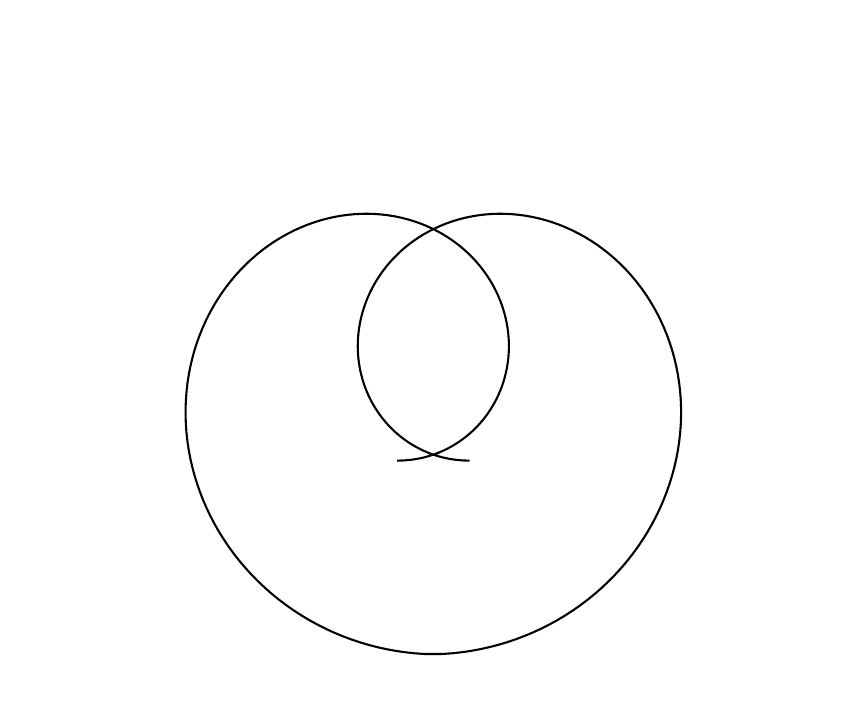}
    \includegraphics[width=0.30\textwidth, angle=0]{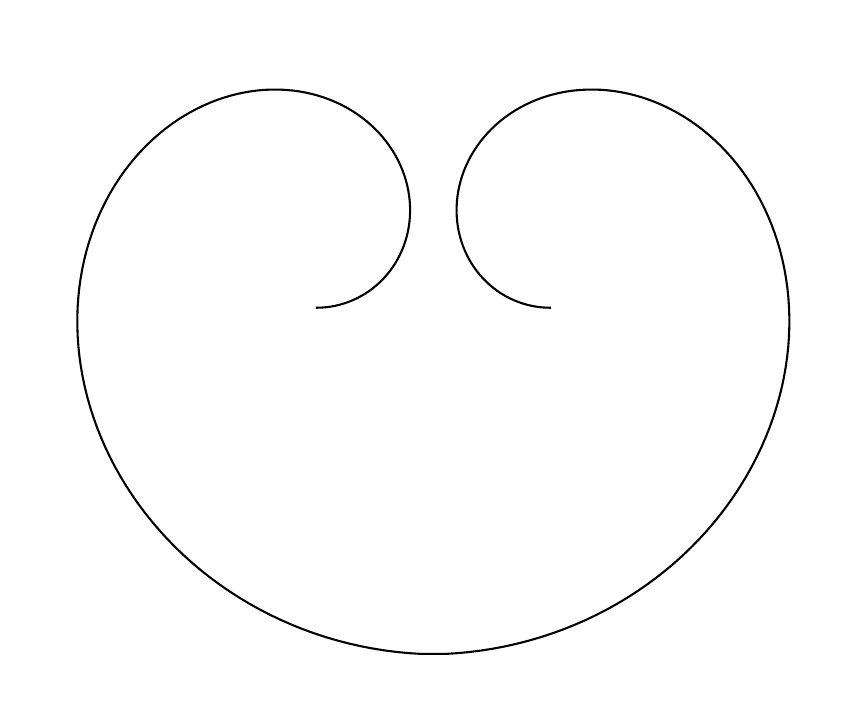}
    \caption{Curves $\Gamma$ corresponding to a full period of the
      solution $\zeta$ with $E = 0.005$, $E = 0.02$ and $E = 0.05$
      respectively (left to right).  The curves appear not to be
      periodic, and to complete an angle of $4\pi$.}
    \label{fig:data}
  \end{figure}
\end{rmk}

\section{Computing $X_{1}$}
Since $A=Z_{\El, 1}$, formulae \eqref{E:X1} and \eqref{E:5.1} can be exploited to find $X_1$. Indeed, let us recall that
\begin{align*}
  Z_{\El, 1}(\xi)=\Aang(\xi) = \frac{1}{15}\int_0^\xi\frac{\rrr''(\xi')}{\rrr(\xi')} d\xi' + \frac{1}{15\Cl^2} \int_0^{\xi} \rrr(\xi')^2 d\xi' - \frac{1}{15}\cbv\xi,
\end{align*}
where all derivatives are with respect to Lazutkin. First let us substitute $\xi' = X_0(s')=C_{\El} \int_0^{s'} \rho_{\al}^{-2/3}(\zeta)d\zeta$, $d\xi'=C_\El\rho_{\al}^{-2/3}(s')ds'$ to obtain
\begin{align*}
Z_{\al,1}(X_0(s))=\frac{1}{15} \int_0^s C_\El \rho_{\al}^{-2/3}(s') \frac{\rrr''(X_0(s'))}{\rrr(X_0(s'))} ds'\\
 + \frac{1}{15\Cl} \int_0^{s} \rho_{\al}^{-2/3}(s') \rrr(X_0(s'))^2 ds'
 - \frac{1}{15}\cbv X_0(s),
\end{align*}
where $s$ is such that $X_0(s)=\xi$, and put $h=\frac{\rho_\El^{-1/3}}{2\sqrt{2}}$ to obtain
\begin{align*}
  Z_{\al, 1}(X_0(s)) =  \frac{4 \Cl}{135} \int_0^s \frac{\rhoL'(X_0(s'))^2}{\rho_{\al}(s')^{8/3}}ds'
  - \frac{\Cl}{45} \int_0^s \frac{\rhoL''(X_0(s'))}{\rho_{\al}(s')^{5/3}}ds'\\
  +\frac{1}{120 \Cl} \int_0^s \frac{ds'}{\rho_{\al}(s')^{4/3}} - \frac{\cbv}{15} X_0(s).
\end{align*}
It can be easily find that
$$\rhoL'(X_0(s))=\frac{\rho_{\al}^{2/3}(s)}{\Cl}\dot{\rho}_{\al}(s),$$
$$
\rhoL''(X_0(s))=\frac{1}{\Cl^2} \rho_{\al}^{4/3}(s)\ddot{\rho}_{\al}(s) +\frac{2}{3\Cl^2} \rho_{\al}^{1/3}(s) \dot{\rho}_{\al}(s)^2,
$$
where dots represent differentiation with respect to arc-length. Thus
\begin{align*}
Z_{\al, 1}(X_0(s)) = \frac{2}{135} \int_0^s \frac{\dot{\rho}_{\al}(s')^2}{\Cl \rho_{\al}(s')^{4/3}}ds'
-\frac{1}{45} \int_0^s \frac{\ddot{\rho}_{\al}(s')}{\Cl \rho_{\al}(s')^{1/3}}ds'\\
+\frac{1}{120 \Cl} \int_0^s \frac{ds'}{\rho_{\al}(s')^{4/3}} - \frac{\cbv}{15} X_0(s).
\end{align*}

By \eqref{E:X1} we have $X_1(s)= - Z_{\al, 1}(X_0(s)) \dot{X}_0(s) Y_0^2(s)$, thus finally we obtain
\begin{align*}
X_1(s)= -\int_0^s \bigg[ \frac{8 \Cl^2 \dot{\rho}_{\al}(s')^2}{135 \rho_{\al}(s')^{4/3}}
-\frac{{4 \Cl^2\ddot{\rho}_{\al}(s')}}{45 \rho_{\al}(s')^{1/3}}
+ \frac{\Cl^2}{30 \rho_{\al}(s')^{4/3}} - \frac{4 \Cl^3\cbv}{15 \rho_{\al}^{2/3}(s')}  \bigg]ds'.
\end{align*}



\appendix
\section{Derivatives}
\newcommand{\rhol}{\rho_{\El}}
\begin{lem}
\label{L:appendix}
Let $\Omega$ be an arbitrary convex domain with a $C^8$ boundary, $x$
denote the Lazutkin parametrization of its boundary.  For $x\in\bT$,
recall that $\rhol(x)$ denotes the radius of curvature of the boundary
at the point on $\partial\Omega$ parametrized by $x$ and
$\theta_{\El}(x)$ the angle between the tangent at the same point and
the positive horizontal semi-axis.  Further, let $\Aang\in C^{3}(\bT)$
and
\begin{align*}
  x_0 = \xi+\Aang(\xi)\eta^{2}, \quad x_+ = \xi + \eta + {\Aang(\xi + \eta)}\eta^{2}.
\end{align*}
Then there exist functions $S_{j}$ and $C_{j}$ for $j = 0,\cdots,5$ so
that:
\begin{align}\label{eq:taylor-expansion}
  \int_{x_0}^{x_+} e^{ i (\theta_{\El}(x')-\theta_{\El}(x_{0}))}\rhol(x')^{2/3}dx' =
  \sum_{j = 0}^{5} C_j\eta^{j}+ i\sum_{j = 0}^{5} S_j\eta^{j}+O(\eta^{6})
\end{align}
and the following formulas hold: let us renormalize
\begin{align*}
  S_{j} &= \frac{\rho_{\El}^{1/3}}{C_{\El}^{2}} \tilde S_{j}&
  C_{j} &= \frac{\rho_{\El}^{2/3}}{C_{\El}}\tilde C_{j}
\end{align*}
Then we have:
\begin{align*}
  \tilde S_{0} &= \tilde S_{1} = 0&
  \tilde S_{2} &= \frac12&
  \tilde S_{3} &= \frac16\frac{\rhol'}{\rhol}
\end{align*}
\begin{align*}
  \tilde S_{4} &=
                 \Aang'-\frac1{27}\left(\frac{\rhol'}{\rhol}\right)^2+\frac5{27}\frac{\rhol''}{\rhol}
                 -\frac1{24C_{\El}^{2}}\rhol^{-2/3}.
\end{align*}
\begin{align*}
  \tilde S_{5} &= \frac12\Aang'' + \frac13 \Aang'\frac{\rhol'}{\rhol}+\frac2{135}\left(\frac{\rhol'}{\rhol}\right)^{3}-\frac4{135}\frac{\rhol'\rhol''}{\rhol^{2}}+\frac7{360}\frac{\rhol'''}{\rho} -\frac1{180C_{\El}^{2}}\frac{\rhol'}{\rhol^{5/3}},
\end{align*}
and
\begin{align*}
  \tilde C_{0} &= 0& \tilde C_{1} &= 1 &\tilde C_{2} = \frac13\frac{\rhol'}{\rhol}
\end{align*}
\begin{align*}
  \tilde C_{3} = \Aang'-\frac1{27}\left(\frac{\rhol'}{\rhol}\right)^2
  + \frac1{9}\frac{\rhol''}{\rhol}
                 -\frac1{6 C_{\El}^{2}}\rhol^{-2/3}
\end{align*}
\begin{align*}
  \tilde C_{4} &= \frac12\Aang'' + \frac13
                 \frac{\rhol'}{\rhol}\Aang' +
                 \frac1{81}\left(\frac{\rhol'}{\rhol}\right)^{3} -
                 \frac1{36}\frac{\rhol'\rhol''}{\rhol^{2}}+\frac1{36}\frac{\rhol'''}{\rho}
                 -\frac1{24 C_{\El}^{2}}\frac{\rhol'}{\rhol^{5/3}}.
\end{align*}
\end{lem}
\begin{proof}
  Integrating~\eqref{eq:theta-rho-relation-lazutkin} yields:
\begin{align*}
  \theta_{\El}(x')-\theta_{\El}(x_{0}) =
  \int_{x_{0}}^{x'}\theta_{\El}'(x'')dx''= C_{\El}^{-1}\int_{x_{0}}^{x'}\rhol^{-1/3}(x'') dx''.
\end{align*}
We can thus write~\eqref{eq:taylor-expansion} as:
\begin{align*}
  \int_{x_0}^{x_+} e^{ i
  (\theta_{\El}(x')-\theta_{\El}(x_{0}))}\rhol(x')^{2/3}dx' &=
\int_{x_0}^{x_+} \exp\left(\frac{i}{C_{\El}}\int_{x_{0}}^{x'}\rhol^{-1/3}(x'') dx''\right)\rhoL(x')^{2/3}dx'.
\end{align*}
The result then follows by performing a formal Taylor expansion in
$\eta$ around $\eta = 0$ and carrying out the algebra.
\end{proof}

\bibliographystyle{abbrv}
\bibliography{alpha}
\end{document}